\documentclass[reqno,12pt]{amsart}
\usepackage[margin=1.1in, footskip=1cm]{geometry}
\usepackage{amsmath, amsthm, amssymb, booktabs}
\usepackage{mathrsfs}
\pdfpagewidth=\paperwidth
\pdfpageheight=\paperheight

\newtheorem{theorem}{Theorem}[section]
\newtheorem{lemma}[theorem]{Lemma}
\newtheorem{corollary}[theorem]{Corollary}

\theoremstyle{definition}

\theoremstyle{remark}

\numberwithin{equation}{section}

\newcommand{\mmod}[1]{\,\,({\rm{mod}}\,\,#1)}

\def\bfx{{\mathbf x}}

\def\dbZ{{\mathbb Z}}

\def\grm{{\mathfrak m}}\def\grM{{\mathfrak M}}\def\grN{{\mathfrak N}}

 \def\Del{{\Delta}}

\def\Ups{{\Upsilon}}

\def\d{{\partial}}
\def\eps{\varepsilon}

\def\le{\leqslant} \def\ge{\geqslant}

\def\d{{\,{\rm d}}}
\def\mdiv{{\,|\,}}

\def\kp{\kappa} 
\def\ww{t}      

\makeatletter
\@namedef{subjclassname@2020}{\textup{2020} Mathematics Subject Classification}
\makeatother

\begin{document}
\title[Smooth Weyl sums]{Estimates for smooth Weyl sums on major arcs}
\author[J\"org Br\"udern]{J\"org Br\"udern}
\address{Mathematisches Institut, Bunsenstrasse 3--5, D-37073 G\"ottingen, Germany}
\email{jbruede@gwdg.de}
\author[Trevor D. Wooley]{Trevor D. Wooley}
\address{Department of Mathematics, Purdue University, 150 N. University Street, West 
Lafayette, IN 47907-2067, USA}
\email{twooley@purdue.edu}
\thanks{First author supported by Deutsche Forschungsgemeinschaft Project Number 462335009. Second author supported by NSF grants DMS-1854398 and DMS-2001549. 
The second author is grateful to Institut Mittag-Leffler for excellent working conditions which assisted the completion of this memoir.}
\date{}

\begin{abstract} We present estimates for smooth Weyl sums of use on sets of major arcs in applications of the Hardy-Littlewood method. In particular, we derive mean 
value estimates on major arcs for smooth Weyl sums of degree $k$ delivering essentially optimal bounds for moments of order $u$ whenever $u>2\lfloor k/2\rfloor +4$.
\end{abstract}

\maketitle

\section{Introduction} The introduction by Vaughan \cite{Vau1989a,Vau1989b} of smooth numbers into the armoury of the circle method practitioner created a flexible new 
tool for the analysis of diagonal problems involving $k$-th powers, and indeed the sharpest bounds currently available in Waring's problem require the use of smooth Weyl 
sums (see \cite{BW2023, VW1995, VW2000, Woo1992, Woo1995, Woo2016}, for example). While powerful mean value estimates for these modified Weyl sums are available 
for handling the minor arc contributions in applications of the circle method, the absence of correspondingly powerful major arc estimates is a source of technical 
complication with the potential to block certain applications. Our goal in this note is to derive estimates for smooth Weyl sums of major arc type sufficiently powerful that 
the vast majority of technical complications may be avoided. Indeed, our recent work on sums of squares and higher powers has already made use of these new major arc 
estimates (see the discussion surrounding \cite[equation (3.15)]{BW2024}).\par

In order to proceed further we must introduce some notation. We employ the letter $p$ to denote a prime number, and as usual, we abbreviate 
${\mathrm e}^{2\pi \mathrm iz}$ to $e(z)$. When $P$ and $R$ are real numbers with $2\le R\le P$,  we write
\[
\mathscr A(P,R)=\{ n\in [1,P]\cap \dbZ :p\mdiv n\Rightarrow p\le R\} .
\]
Also, when $\nu>1$, we put $\mathscr A_\nu (P,R)=\mathscr A(P,R)\cap (P/\nu ,P]$. In addition, fixing an integer $k\ge 2$, we define the smooth Weyl sums
\[
g_\nu(\alpha;P,R)=\sum_{x\in \mathscr A_\nu (P,R)}e(\alpha x^k)\quad \text{and}\quad g(\alpha;P,R)=\sum_{x\in \mathscr A(P,R)}e(\alpha x^k).
\]
Finally, we define the multiplicative function $\kp(q)=\kp_k(q)$ by putting
\[
\kp(p^{uk+v})=\begin{cases} kp^{-u-1/2},&\text{when $u\ge 0$ and $v=1$,}\\
p^{-u-1},&\text{when $u\ge 0$ and $2\le v\le k$.}\end{cases}
\]
We note that, as in \cite[Lemma 3]{Vau1986}, when $k\ge 3$ one has $q^{-1/2}\le \kp(q)\ll q^{-1/k}$. When $k=2$ we have the slightly weaker bounds 
$q^{-1/2}\le \kp(q)\le 2^{\omega(q)}q^{-1/2}$. Here $\omega(q)$ is the number of distinct primes dividing $q$, and the implicit constant in Vinogradov's notation 
depends on $k$.\par

In \S2 we establish the estimates for the exponential sums $g_\nu(\alpha)=g_\nu(\alpha;P,R)$ and $g(\alpha)=g(\alpha;P,R)$ contained in the following theorem. Here and 
throughout, we find it useful to introduce the notation
\begin{equation}\label{1.0}
L=\log P,\quad L_2=\log \log (3R)\quad \text{and}\quad \mathscr  L=\log (2+P^k|\alpha -a/q|).
\end{equation}

\begin{theorem}\label{theorem1.1}
Let $k$ be a natural number with $k\ge 2$, and let $\nu$ and $\varepsilon$ be real numbers with $\nu>1$ and $\varepsilon>0$. Suppose that $R$ and $P$ are real 
numbers with $2\le R\le P$. Then, whenever $\alpha \in \mathbb R$, $a\in \mathbb Z$ and $q\in \mathbb N$ satisfy $(a,q)=1$, one has
\[
g_\nu (\alpha)\ll \frac{q^\varepsilon \kp(q)^{1/2}PL(L_2\mathscr L)^{1/2}}{(1+P^k|\alpha -a/q|)^{1/2}}+q^\varepsilon P^{3/4}R^{1/2}(LL_2)^{1/4}
\bigl( q+P^k|q\alpha -a|\bigr)^{1/8}.
\]
Moreover, under the same conditions, one has
\[
g(\alpha)\ll \frac{q^\varepsilon \kp(q)^{1/2}PLL_2^{1/2}}{(1+P^k|\alpha -a/q|)^{1/k}}+
q^\varepsilon P^{3/4}R^{1/2}(LL_2)^{1/4}\bigl( q+P^k|q\alpha -a|\bigr)^{1/8}\quad (k\ge 3)
\]
and
\[
g(\alpha)\ll \frac{q^\varepsilon \kp(q)^{1/2}PLL_2^{1/2}\mathscr L^{3/2}}{(1+P^2|\alpha -a/q|)^{1/2}}+
q^\varepsilon P^{3/4}R^{1/2}(LL_2)^{1/4}\bigl( q+P^2|q\alpha -a|\bigr)^{1/8}
\quad (k=2).
\]
\end{theorem}

We note that when $1+P^k|\alpha -a/q|>P^2$, the estimates supplied by this theorem are worse than trivial. There is thus no loss of generality in assuming henceforth that 
$1+P^k|\alpha -a/q|\le P^2$, so that
\begin{equation}\label{1.1z}
\mathscr L=\log \bigl( 2+P^k|\alpha -a/q|\bigr) \ll L.
\end{equation}

\par In order to put into context previous work on this subject, we introduce a fairly general Hardy-Littlewood dissection of the unit interval. Let $Q$ be a real number with 
$1\le Q\le P^{k/2}$. When $q\in \mathbb N$ and $a\in \mathbb Z$, we take
\[
\grM(q,a;Q) =\{\alpha \in [0,1):  |q\alpha -a|\le QP^{-k}\}.
\]
We write $\grM(Q)$ for the union of the intervals $\grM(q,a;Q)$ over coprime integers $a$ and $q$ with $0\le a\le q\le Q$, and we put $\grm(Q)=[0,1)\setminus \grM(Q)$. 
In applications of the circle method, the most frequently encountered model Hardy-Littlewood dissection is that into major arcs $\grM(P)$ and minor arcs $\grm(P)$ thus 
defined. On noting that whenever $\alpha\in \grM(Q)$, one has $q+P^k|q\alpha -a|\le 2Q$, we discern from Theorem \ref{theorem1.1} and the bound \eqref{1.1z} that 
whenever $\alpha\in \grM(q,a;P)\subseteq \grM(P)$, one has
\begin{equation}\label{1.1}
g_\nu(\alpha;P,R)\ll \frac{q^\varepsilon \kp(q)^{1/2}P(\log P)^{3/2+\varepsilon }}{(1+P^k|\alpha -a/q|)^{1/2}}+P^{7/8+\varepsilon}R^{1/2}.
\end{equation}
In the special case $k=3$, a related conclusion is available from \cite[Lemma 2.2]{BW2001}, though with the exponent $9/10$ in place of $7/8$. We note, however, that the 
estimate obtained in \cite{BW2001} is applicable for a wider set of parameters $\alpha$. Rather than attempt to generalise the argument underlying the proof of this lemma 
to general exponents $k$, which would in any case yield substantially weaker estimates than we derive in \eqref{1.1}, we instead turn to earlier work of the second author 
joint with Vaughan \cite{VW1991} for inspiration. As enhanced in \cite[Lemma 5.4]{PW2002}, this approach shows that whenever $2\le R\le M\le P$ and 
$q+P^k|q\alpha -a|\le TM$, then
\[
g(\alpha;P,R)\ll q^\varepsilon (\log P)^3\Bigl( \frac{P}{(q+P^k|q\alpha -a|)^{1/(2k)}}+(PMR)^{1/2}+PR^{1/2}\Bigl( \frac{T}{M}\Bigr)^{1/4}\Bigr) .
\]
Experts will recognise that a factor $q^{-1/(2k)}$ in this estimate is replaced in Theorem \ref{theorem1.1} by a factor $\kp(q)^{1/2}$, an enhancement that offers 
substantially stronger control of mean value estimates restricted to sets of major arcs. There is also a useful improvement in the dependence on $1+P^k|\alpha -a/q|$.\par

Very recently, using methods quite different from those that we apply here, Shparlinski \cite[Theorem 1.3]{Shpa} estimated $g(\alpha;P,R)$ in the special case where the 
argument is a rational number with prime denominator $p$. In his work, it is assumed that $p> (4P)^{3/4}$. Then, for $a\in \mathbb Z$ with $(a,p)=1$ and 
$2\le R\le p^{1/2}$, it is shown that for each $\varepsilon >0$ one has
\[
g(a/p;P,R) \ll p^{5/48+\varepsilon}P^{29/32}R^{3/104}.
\]
Note that this estimate is worse than the trivial bound when $p>P^{9/10}$ so that we may concentrate on the situation with $(4P)^{3/4}<p\le P^{9/10}$. In this range, 
Theorem \ref{theorem1.1} supplies the estimate
\[
g(a/p;P,R) \ll Pp^{\varepsilon-1/2}+P^{3/4+\varepsilon}p^{1/8}R^{1/2} \ll P^{3/4+\varepsilon}p^{1/8}R^{1/2} .
\]
In the most interesting case where $R$ is a very small power of $P$, our bound is decidedly smaller, and our bound is non-trivial for $p$ almost as large as $P^2$.\par

We turn our attention now to the factor $\log P$ occurring in the first term on the right hand side of our estimate for $g_\nu(\alpha)$ in Theorem \ref{theorem1.1}.~This has 
the potential to prevent the immediate use of this theorem in pursuit of near optimal bounds for moments of smooth Weyl sums over major arcs of the shape $\grM(Q)$, 
when $Q$ is no larger than a modest power of $\log P$. In \S3 we obtain estimates for smooth Weyl sums that eliminate this deficiency. The bounds feature Euler's totient, 
denoted $\phi(q)$, as well as the closely associated multiplicative function
\[
\psi(q)=q/\phi(q)=\prod_{p\mdiv q}(1-1/p)^{-1}.
\]
We note for future reference that $\psi(q)=O(\log \log q)$ for $q\ge 3$ (see \cite[Theorem 2.9]{MV2007}, for example).

\begin{theorem}\label{theorem1.2}
Let $k$ be a natural number with $k\ge 2$, and let $\nu$ be a real number with $\nu>1$. Suppose that $R$ and $P$ are real numbers with $2\le R\le P^{1/2}$. Then, for each 
$A\ge 1$, there is a positive number $c$ with the following property. Whenever $a\in \mathbb Z$, $q\in \mathbb N$ and $\alpha \in \mathbb R$ satisfy $(a,q)=1$ and 
$q\le (\log P)^A$, one has
\[
g_\nu (\alpha;P,R)\ll \frac{\kp(q)\psi(q)P}{1+P^k|\alpha -a/q|}+P\exp \bigl( -c(\log P)^{1/2}\bigr) \bigl( 1+P^k|\alpha -a/q|\bigr)
\]
and
\[
g(\alpha;P,R)\ll \frac{\kp(q)\psi(q)P}{(1+P^k|\alpha -a/q|)^{1/k}}+P\exp \bigl( -c(\log P)^{1/2}\bigr) \bigl( 1+P^k|\alpha -a/q|\bigr).
\]
\end{theorem}

By way of comparison, the earlier work of the second author joint with Vaughan \cite[Lemma 8.5]{VW1991} provides a conclusion of similar type, though with the factor 
$\kp(q)$ replaced by $q^{\varepsilon-1/k}$. While this previous approach certainly addresses the issue raised in the preamble to Theorem \ref{theorem1.2}, our new 
conclusion permits control of major arc moments in which the number of generating functions can be very nearly halved.\par

Following the proof of an auxiliary pruning lemma in \S4, we obtain estimates for mean values of the smooth Weyl sum $g_\nu (\alpha;P,R)$ on sets of major arcs in \S5. 
These estimates provide near optimal upper bounds for corresponding moments of the exponential sum $g(\alpha ;P,R)$.

\begin{theorem}\label{theorem1.3}
Let $k$ and $t$ be natural numbers with $k\ge 3$ and $t\ge \lfloor k/2\rfloor$. Furthermore, let $\omega$ and $\omega'$ be positive numbers with
\[
\omega <\frac{2t+4}{t+10}\quad \text{and}\quad \omega'<\frac{2k}{k+4},
\]
and put $\Omega=\min\{ \omega,\omega'\}$. Then, for any $\varepsilon>0$ there is a real number $\eta>0$ with the property that whenver $2\le R \le P^\eta$, one has
\begin{equation} \label{1.2}
\int_{\grM(P^\omega)}|g(\alpha ;P,R)|^{2t+4}\d\alpha \ll P^{2t+4-k+\varepsilon}.
\end{equation}
Moreover, when $u$ is a real number with $u>2t+4$, then provided that
\[
1\le Q\le P^{\Omega}\quad \text{and}\quad \tau<(u-2t-4)/(2k),
\]
one has
\[
\int_{\grM(P^{\Omega})\setminus \grM(Q)}|g(\alpha ;P,R)|^u\d\alpha \ll P^{u-k}Q^{-\tau}.
\]
\end{theorem}

This theorem delivers the near optimal estimate \eqref{1.2}. Moreover, one may allow the exponent $\omega$ to approach $2$ as $t\rightarrow \infty$. By contrast, earlier 
treatments would restrict $\omega$ to satisfy the constraint $0\le \omega \le 1$. Furthermore, when $k\ge 4$ one has $t\ge 2$, and in such circumstances $\omega$ can be 
nearly as large as $2/3$. One finds that $\omega$ can be taken to be any number with $0\le \omega <1$ as soon as $t\ge 6$, and this is assured when $k\ge 12$. When 
$k=3$ and $t=1$, the estimate \eqref{1.2} was demonstrated in \cite[Corollary 3.2]{BW2001} with $\omega <2/5$. Our new result covers the range $\omega <6/11$.\par

We next record a corollary to Theorem \ref{theorem1.3} that is employed in our recent work \cite{BW2024} on the representation of integers as sums of a square and a 
number of $k$-th powers. Here, and throughout, we define
\begin{equation}\label{1.3}
\grN(Q)=\grM(Q)\setminus \grM(Q/2).
\end{equation}

\begin{corollary}\label{corollary1.4}
Let $k$ be a natural number with $k\ge 3$, put $t=\lfloor k/2\rfloor$, and suppose that $u$ is a real number with $u>2t+4$. Furthermore, let $\tau$ be a real number with 
$0<\tau<(u-2t-4)/(2k)$. Then, there is a real number $\eta>0$ with the property that whenever $2\le R \le P^\eta$ and $1\le Q\le P^{1/2}$, one has
\[
\int_{\grN(Q)}|g(\alpha ;P,R)|^u\d\alpha \ll P^{u-k}Q^{-\tau}.
\]
\end{corollary}

We finish this memoir in \S6 with an account of the consequences of our new estimates for Waring's problem. We are able to establish lower bounds for the contribution of 
the major arcs in Waring's problem of the expected order, provided that the number of variables available is at least $2\lfloor k/2\rfloor+5$, when $k\ge 3$.  In the classical 
situation in which variables are not restricted, and in particular are not restricted to be smooth, analogous conclusions would require at least $k+2$ variables. Previous work 
using smooth Weyl sums, meanwhile, would require at least $2k+3$ variables. The superiority of our new conclusions is clear. We refer the reader to Theorems 
\ref{theorem7.1} and \ref{theorem7.2} for details of these new results.\par

In this memoir, our basic parameter is $P$, a sufficiently large positive number. We adopt the usual convention that, whenever a statement involves the letter $\varepsilon$, 
then it is asserted that the statement holds for any positive value of $\varepsilon$. Implicit constants in the notations of Vinogradov and Landau may depend on 
$\varepsilon$, as well as ambient parameters implicitly fixed, such as $k$ and $u$. We write $p^h\|n$ to denote that $p^h\mdiv n$, but $p^{h+1}\nmid n$. 

\section{An enhanced major arc estimate}
Our goal in this section is the proof of Theorem \ref{theorem1.1}. In order to establish the bounds for $g_\nu(\alpha;P,R)$ recorded therein, we adapt the analysis of 
\cite[\S7]{VW1991}, taking advantage of more precise auxiliary estimates deemed unnecessary in this earlier treatment. Our account, nonetheless, follows closely the 
argument of the proof of \cite[Lemma 7.2]{VW1991}. Throughout, in order to ease our exposition, we make use of the notation recorded in \eqref{1.0}. We begin with a 
lemma concerning the multiplicative function $\kp(q)$ required later in this section.

\begin{lemma}\label{lemma2.1} Let $q\in \mathbb N$ and suppose that $d\mdiv q$. Let $e_0$ be a divisor of $d$, and write $d/e_0=d_1d_2^2\cdots d_k^k$, where 
$d_1,\ldots ,d_{k-1}$ are pairwise coprime and squarefree. Then
\[
\kp(q/d)\le k^{4\omega(q)}\kp(q)e_0d_1\cdots d_k.
\]
\end{lemma}

\begin{proof} We claim that whenever $a$ and $b$ are non-negative integers with $a+b\ge 1$, then
\begin{equation}\label{2.1}
\kp(p^a)\kp(p^b)\le k^2\kp(p^{a+b}).
\end{equation}
In order to establish this assertion, we note that in view of the definition of $\kp(q)$, it suffices to consider the situation in which $1\le a,b\le k$. We divide into three cases. 
When $2\le a,b\le k$, we have
\[
\kp(p^a)\kp(p^b)=p^{-2}\quad \text{and}\quad \kp(p^{a+b})\ge p^{-2}.
\]
When $a=1$ and $2\le b\le k$, and also when $b=1$ and $2\le a\le k$, we have
\[
\kp(p^a)\kp(p^b)=kp^{-3/2}\quad \text{and}\quad \kp(p^{a+b})\ge p^{-3/2}.
\]
Finally, when $a=b=1$, one has
\[
\kp(p^a)\kp(p^b)=k^2p^{-1}\quad \text{and}\quad \kp(p^{a+b})=p^{-1}.
\]
This confirms the upper bound \eqref{2.1} in  all cases.\par

The multiplicative property of $\kp(q)$ now confirms that whenever $q\in \mathbb N$ and $d\mdiv q$, one has
\[
\kp(d)\kp(q/d)\le k^{2\omega(q)}\kp(q),
\]
and similarly
\[
\kp(e_0)\kp(d_1d_2^2\cdots d_k^k)\le k^{2\omega(d)}\kp(d).
\]
But $\kp(e_0)\ge e_0^{-1}$ and $\kp(d_1d_2^2\cdots d_k^k)\ge (d_1d_2\cdots d_k)^{-1}$. Hence, we conclude that
\[
(e_0d_1d_2\cdots d_k)^{-1}\kp(q	/d)\le k^{2\omega(d)+2\omega(q)}\kp(q)\le k^{4\omega(q)}\kp(q).
\]
The conclusion of the lemma follows immediately.
\end{proof}

Next, we provide two simple estimates for  auxiliary sums.

\begin{lemma}\label{lemma2.2}
Suppose that $X$ and $J$ are positive numbers. Then
\[
\sum_{1\le j\le J}(1+jX)^{-1}\le 2J(1+XJ)^{-1}\bigl( 1+\log (1+JX)\bigr) .
\]
\end{lemma}

\begin{proof} When $X$ and $J$ are positive numbers with $X\le J^{-1}$, then
\[
\sum_{1\le j\le J}(1+jX)^{-1}\le J\le 2J(1+XJ)^{-1}.
\]
When instead $J^{-1}<X$, then
\begin{align*}
\sum_{1\le j\le J}(1+jX)^{-1}&\le X^{-1}\bigl( 1+\log (1+JX))\\
&\le 2J(1+XJ)^{-1}\bigl( 1+\log (1+JX)\bigr) .
\end{align*}
The desired conclusion therefore follows in both cases.
\end{proof}

\begin{lemma}\label{lemma2.3}
Let $k$ and $J$ be integers with $k\ge 2$ and $1\le 2^J \le P$, and let $\lambda$ be a real number with $1/k<\lambda\le 1$.  When $j$ is a non-negative integer, put 
$X_j=2^{-j}P$. Then, for all real numbers $\beta$, one has
\begin{equation}\label{2.1a}
\sum_{0\le j\le J} X_j(1+X_j^k|\beta|)^{-\lambda}\ll \frac{P}{(1+P^k|\beta|)^{1/k}}
\end{equation}
while
\[
\sum_{0\le j\le J} X_j(1+X_j^k|\beta|)^{-1/k} \ll \frac{P\log (2+P^k|\beta|)}{(1+P^k|\beta|)^{1/k}}.
\]
\end{lemma}

\begin{proof} Let $K$ denote the sum on the left hand side of \eqref{2.1a}, and suppose first that $\lambda > 1/k$.  If $|\beta|\le P^{-k}$, then 
\[
K \le \sum_{j\ge 0} 2^{-j}P \le 2P,
\]
which is acceptable.
If $|\beta|\ge 1$, then since $\lambda k>1$, we find that
\[ K  \le \sum_{2^j\le P} X_j^{1-\lambda k}|\beta|^{-\lambda} = P^{1-\lambda k}|\beta|^{-\lambda} \sum_{2^j\le P} 2^{j(\lambda k-1)} \ll |\beta|^{-\lambda},
\]
which is superior to our claim. We may now suppose that $P^{-k} \le |\beta| \le 1$. Then, there is an integer $I$ with
\[
2^{-I}P\ge |\beta|^{-1/k} \ge 2^{-I-1}P.
\]
When $j\le I$, one has $X_j^k|\beta| \ge 1$, whence the contribution of these values of $j$ to $K$ is bounded by
\[
|\beta|^{-\lambda} \sum_{j\le I} X_j^{1-\lambda k} \ll P^{1-\lambda k}|\beta|^{-\lambda} \sum_{2^j\le P|\beta|^{1/k}} 2^{j(\lambda k-1)} \ll |\beta|^{-1/k}.
\]
For the contribution of the summands in $K$ where $j>I$ we have the upper bound
\[  \sum_{j>I} 2^{-j} P\ll 2^{-I} P\ll  |\beta|^{-1/k}. \]
This proves the first claim of the lemma.\par

To prove the second claim, one runs the above argument with $\lambda=1/k$, in which circumstance one has $\lambda k-1 =0$. The estimate $K\le 2P$ remains valid for 
$|\beta|\le P^{-k}$. When $|\beta|\ge 1$, meanwhile, we find that
\[
K\le |\beta|^{-1/k}\sum_{2^j\le P}1\ll |\beta|^{-1/k}\log P,
\]
which is again acceptable. When $P^{-k}\le |\beta|\le 1$, we choose the integer $I$ as in the first case, observing that we now have
\[
K\le \sum_{1\le j\le I}|\beta|^{-1/k}+\sum_{j>I}2^{-j}P\ll |\beta|^{-1/k}\log (P^k|\beta|)+|\beta|^{-1/k}.
\]
The desired conclusion therefore follows also in this final situation.\end{proof}

We are now equipped to prove our first estimate for $g_\nu(\alpha;P,R)$. In this context, we remind the reader of the notation introduced in \eqref{1.0}.

\begin{lemma}\label{lemma2.4}
Let $k$ be a natural number with $k\ge 2$, and let $\nu$ be a real number with $\nu>1$. Suppose that $R$, $P$ and $M$ are real numbers with 
\[
2\le R\le M\le P/\nu.
\]
Then, whenever $\alpha \in \mathbb R$ and $a\in \mathbb Z$, $q\in \mathbb N$ are coprime, one has
\begin{align}
g_\nu(\alpha ;P,R)\ll &\, \frac{ q^\varepsilon \kp(q)^{1/2}PL(L_2\mathscr L)^{1/2}}{(1+P^k|\alpha -a/q|)^{1/2}}+q^\varepsilon (LL_2)^{1/2}(PMR)^{1/2}
\notag \\
&\ \ \ \ \ \ \ \ \ +q^\varepsilon P(R/M)^{1/2}\bigl( q+P^k|q\alpha -a|\bigr)^{1/4}.\label{2.2}
\end{align}
\end{lemma}

\begin{proof} In this discussion, we provide certain details of the account of the proof of \cite[Lemma 7.2]{VW1991} that could be suppressed at the cost of opacity of 
exposition. We begin with an initial decomposition of the exponential sum $g_\nu(\alpha;P,R)$ into pieces equipped with a bilinear structure. For each prime number $\pi$ 
with $\pi\le R$, we define
\[
\mathscr B_\pi=\{ v\in \mathscr A(M\pi,R):\text{$v>M$, $\pi\mdiv v$, and $p\mdiv v\Rightarrow p\ge \pi$}\}.
\] 
Then, as a consequence of \cite[Lemma 10.1]{Vau1989a}, one has
\begin{equation}\label{2.4}
g_\nu(\alpha;P,R)=\sum_{\pi \le R}\sum_{v\in \mathscr B_\pi}\sum_{u\in \mathscr A_\nu (P/v,\pi)}e(\alpha (uv)^k).
\end{equation}

\par Next, write
\[
\mathscr W=\{ (2^j,2^iM):\text{$i\ge 0$, $j\ge -1$, $2^i<R$ and $P/\nu <2^{i+j}M<P$}\},
\]
and define
\[
S_1(U,V)=\sum_{\pi\le R}\sum_{\substack{V<v\le 2V\\ v\in \mathscr B_\pi}}\biggl| \sum_{\substack{u\in \mathscr A_\nu (P/v,\pi)\\ U<u\le 2U}}e(\alpha (uv)^k)\biggr| .
\]
Then, much as in the analogous treatment of \cite{VW1991}, we see that
\begin{equation}\label{2.5}
\biggl| \sum_{\pi \le R}\sum_{v\in \mathscr B_\pi}\sum_{u\in \mathscr A_\nu (P/v,\pi)}e(\alpha (uv)^k)\biggr| \le \sum_{(U,V)\in \mathscr W}S_1(U,V).
\end{equation}
We note for future reference that for each pair $(U,V)\in \mathscr W$, one has
\begin{equation}\label{2.6}
M\le V<MR\quad \text{and}\quad P/(\nu V)<U<P/V.
\end{equation}
Furthermore, since the elements of $\mathscr B_\pi$ are integers divisible by $\pi$, all of whose prime divisors are at least as large as $\pi$, it follows from 
\cite[Chapter III.6, Theorem 3]{Ten} that
\begin{equation}\label{2.6a}
\sum_{\pi\le R}\sum_{\substack{V<v\le 2V\\ v\in \mathscr B_\pi}}1\ll \sum_{\pi\le R}\frac{V}{\pi \log \pi}\ll V.
\end{equation}

\par By Cauchy's inequality, one has
\[
S_1(U,V)^2\le \biggl( \, \sum_{\pi\le R}\sum_{\substack{V<v\le 2V\\ v\in \mathscr B_\pi}}1\biggr) \sum_{\pi\le R}\sum_{\substack{V<v\le 2V\\ v\in \mathscr B_\pi}}\biggl| 
\sum_{\substack{u\in \mathscr A_\nu (P/v,\pi)\\ U<u\le 2U}}e(\alpha (uv)^k)\biggr|^2.
\]
Thus, on isolating the implicit diagonal terms, we discern by means of \eqref{2.6a} that
\begin{equation}\label{2.7}
|S_1(U,V)|^2\ll V^2U+V|S_2(U,V)|,
\end{equation}
where
\begin{equation}\label{2.8}
S_2(U,V)=\sum_{\pi\le R}\sum_{\substack{u_1,u_2\in \mathscr A_\nu (P/V,\pi)\\ U<u_1<u_2\le 2U}}T_1(\alpha \pi^k(u_2^k-u_1^k)),
\end{equation}
in which we write
\[
T_1(\beta)=\sum_{\substack{V/\pi<\ww \le 2V/\pi \\ P/(\nu \pi u_1)<\ww\le P/(\pi u_2)}} e(\beta \ww^k).
\]
On collecting together the estimates \eqref{2.4}, \eqref{2.5} and \eqref{2.7}, we may conclude thus far that
\begin{equation}\label{2.9}
g_\nu (\alpha ;P,R)\ll \sum_{(U,V)\in \mathscr W}\bigl( VU^{1/2}+V^{1/2}|S_2(U,V)|^{1/2}\bigr) .
\end{equation}

\par It is at this point that our treatment begins to diverge further from the path laid down in \cite{VW1991}. We again view $S_2(U,V)$, as defined in \eqref{2.8}, as an 
exponential sum over $k$-th powers of $\ww$. For a typical summand on the right hand side of \eqref{2.8}, we write
\begin{equation}\label{2.10}
D=(q,\pi^k(u_2^k-u_1^k)),
\end{equation}
and then put
\begin{equation}\label{2.11}
b=a\pi^k(u_2^k-u_1^k)/D,\quad r=q/D,\quad \beta =\alpha \pi^k(u_2^k-u_1^k).
\end{equation}
Thus, we have
\begin{equation}\label{2.12}
\beta -\frac{b}{r}=\pi^k(u_2^k-u_1^k)\Bigl( \alpha -\frac{a}{q}\Bigr) ,
\end{equation}
with
\[
(b,r)=(a\pi^k(u_2^k-u_1^k),q)/D=(a,q/D)=1.
\]
We now apply \cite[Theorem 4.1]{Vau1997}. Thus, in view of Lemmata 4.3, 4.4, 4.5 and 6.2 of \cite{Vau1997}, we find that
\begin{equation}\label{2.13}
T_1(\beta)\ll \frac{\kp(r)V/\pi}{1+(V/\pi)^k|\beta -b/r|}+r^{1/2+\eps}\bigl( 1+(V/\pi)^k|\beta -b/r|\bigr)^{1/2}.
\end{equation}

\par Next, write
\begin{equation}\label{2.14}
S_3(U,V)=\sum_{\pi\le R}\sum_{U<u_1<u_2\le 2U}\frac{\kp\bigl( q/(q,\pi^k(u_2^k-u_1^k))\bigr)}{\pi (1+V^k(u_2^k-u_1^k)|\alpha -a/q|)} .
\end{equation}
Then on recalling \eqref{2.8}, we deduce from \eqref{2.10}-\eqref{2.13} that
\[
S_2(U,V)\ll VS_3(U,V)+RU^2q^\eps \bigl( q+(UV)^k|q\alpha -a|\bigr)^{1/2}.
\]
On substituting this estimate into \eqref{2.9} and recalling \eqref{2.6}, we infer thus far that
\begin{equation}\label{2.15}
g_\nu(\alpha ;P,R)\ll (PMR)^{1/2}+\sum_{(U,V)\in \mathscr W}VS_3(U,V)^{1/2}+P(R/M)^{1/2}q^\varepsilon \bigl( q+P^k|q\alpha -a|\bigr)^{1/4}.
\end{equation}

\par We next analyse the sum $S_3(U,V)$ defined in \eqref{2.14}. First, by separately considering the contributions arising from summands in which $\pi\mdiv q$ and 
$\pi\nmid q$, and applying the bound $u_2^k-u_1^k\ge (u_2-u_1)U^{k-1}$, we find that
\[
S_3(U,V)\ll L_2q^\varepsilon\sum_{U<u_1<u_2\le 2U}\frac{\kp(q/(q,u_2^k-u_1^k))}{1+U^{k-1}V^k(u_2-u_1)|\alpha -a/q|}.
\]
Here, we have made use of the bound $\kp(\pi^h)\ge k^{-1}\pi^{-1}\kp(\pi^{h-l})$, valid for $0\le l\le k$ and $h\ge l$.\par

Now write $d=(q,u_2^k-u_1^k)$ and $\gamma=|\alpha -a/q|$. Also, put $m=(u_1,u_2)$ and $t_i=u_i/m$ $(i=1,2)$. Then, we obtain the upper bound
\[
S_3(U,V)\ll L_2q^\varepsilon \sum_{d\mdiv q}\kp(q/d)\sum_{1\le m\le 2U}\sum_{\substack{U/m<t_1<t_2\le 2U/m\\ (t_1,t_2)=1\\ d\mdiv m^k(t_2^k-t_1^k)}}
\bigl( 1+U^{k-1}V^km(t_2-t_1)\gamma \bigr)^{-1}.
\]
We follow the analogous argument of \cite{VW1991} once again. For a given pair of integers $d$ and $m$ occurring in the latter sum, we put $d_0=(d,m^k)$ and 
$e_0=d/d_0$. Then we see that
\begin{equation}\label{2.16}
S_3(U,V)\ll L_2 q^\varepsilon \sum_{d\mdiv q}\kp(q/d)\sum_{d_0e_0=d}\sum_{\substack{1\le m\le 2U\\ d_0\mdiv m^k}}S_4(U,V),
\end{equation}
in which we write
\[
S_4(U,V)=\sum_{\substack{U/m<t_1<t_2\le 2U/m\\ (t_1,t_2)=1\\ 
e_0\mdiv (t_2^k-t_1^k)}}\bigl( 1+U^{k-1}V^km(t_2-t_1)\gamma \bigr)^{-1}.
\]
This sum is analysed (under the name $S_5$) in the discussion following \cite[equation (7.8)]{VW1991}, though in the latter case the inner sum has argument with exponent 
$-1/k$ in place of $-1$. This difference in detail does not impact the argument that follows in any material way, and thus we may conclude that
\begin{equation}\label{2.17}
S_4(U,V)\ll q^\varepsilon \sum_{e_1f_1=e_0}\sum_{1\le j\le U/(me_1)}\Bigl( \frac{U}{mf_1}+1\Bigr) \Bigl( 1+U^{k-1}V^kmje_1\gamma \Bigr)^{-1}.
\end{equation}

\par By applying the upper bound supplied by Lemma \ref{lemma2.2} within \eqref{2.17}, and then recalling \eqref{2.6}, we find that
\begin{align*}
S_4(U,V)&\ll q^\varepsilon \sum_{e_1f_1=e_0}\Bigl( \frac{U}{mf_1}+1\Bigr) \Bigl( \frac{U}{me_1}\Bigr) \frac{1+\log (1+(UV)^k\gamma )}{1+(UV)^k\gamma}\\
&\ll \mathscr Lq^\varepsilon \sum_{e_1f_1=e_0}\Bigl( \frac{U^2}{m^2e_0}+\frac{U}{me_1}\Bigr) (1+P^k\gamma )^{-1}.
\end{align*}
Since $e_0\mdiv q$, we therefore infer that
\[
S_4(U,V)\ll \mathscr Lq^{2\varepsilon}\Bigl( \frac{U^2}{m^2e_0}+\frac{U}{m}\Bigr) (1+P^k\gamma)^{-1}.
\]
By substituting this upper bound into \eqref{2.16}, therefore, we deduce that
\[
S_3(U,V)\ll \frac{L_2\mathscr Lq^{3\varepsilon}}{1+P^k\gamma}\sum_{d\mdiv q}\kp(q/d)\sum_{d_0e_0=d}
\biggl( \biggl( \, \sum_{\substack{1\le m\le 2U\\ d_0\mdiv m^k}}\frac{U^2}{m^2e_0}\biggr) +LU\biggr) .
\]
We now write $d_0=d_1d_2^2\cdots d_k^k$, where $d_1,\ldots ,d_{k-1}$ are squarefree and pairwise coprime, and we recall the conclusion of Lemma \ref{lemma2.1}. Then 
we see that
\[
S_3(U,V)\ll LL_2q^{4\varepsilon}U+\frac{L_2\mathscr Lq^{4\varepsilon}\kp(q)}{1+P^k\gamma} T_2(q)U^2,
\]
where
\[
T_2(q)=\sum_{d\mdiv q}\sum_{e_0d_1d_2^2\cdots d_k^k=d}e_0d_1\cdots d_k\sum_{1\le n\le 2U/(d_1\cdots d_k)}\frac{1}{(nd_1\cdots d_k)^2e_0}.
\]
But we have
\[
T_2(q)\ll \sum_{d\mdiv q}\sum_{e_0d_1\cdots d_k\mdiv d}\frac{1}{d_1\cdots d_k}\ll q^\varepsilon ,
\]
and hence we conclude that
\begin{equation}\label{2.18}
S_3(U,V)\ll LL_2q^{4\varepsilon}U+\frac{L_2\mathscr Lq^{5\varepsilon}\kp(q)U^2}{1+P^k\gamma}.
\end{equation}

On substituting \eqref{2.18} into \eqref{2.15}, and again recalling \eqref{2.6}, we arrive at the upper bound
\begin{align*}
g_\nu(\alpha ;P,R)\ll &\, \frac{q^{3\varepsilon}\kp(q)^{1/2}P(\log R)(L_2\mathscr L)^{1/2}}{(1+P^k|\alpha -a/q|)^{1/2}}+
(LL_2)^{1/2}q^{2\varepsilon}(PMR)^{1/2}\\
&\ \ \ \ \ \ \ \ +P(R/M)^{1/2}q^\varepsilon \bigl( q+P^k|q\alpha -a|\bigr)^{1/4}.
\end{align*}
The upper bound \eqref{2.2} asserted in the statement of the lemma now follows.
\end{proof}

We are now equipped to complete the proof of our first theorem. 

\begin{proof}[The proof of Theorem \ref{theorem1.1}]
Equipped with the hypotheses of the statement of the theorem, we begin by noting that when $q+P^k|q\alpha -a|>P^2$, then its conclusion is inferior to the trivial 
estimate $|g_\nu(\alpha ;P,R)|\le P$. We are able to assume henceforth, therefore, that $q+P^k|q\alpha -a|\le P^2$. We may consequently apply Lemma \ref{lemma2.4} 
with
\[
M=(LL_2)^{-1/2}P^{1/2}(q+P^k|q\alpha -a|)^{1/4},
\]
since in these circumstances one has $R<M<P/\nu$ whenever $P$ is sufficiently large. With this choice for $M$, the second and third terms on the right hand side of 
\eqref{2.2} satisfy the bound
\begin{align*}
q^\varepsilon (LL_2)^{1/2}(PMR)^{1/2}+q^\varepsilon P(R/M)^{1/2}&(q+P^k|q\alpha -a|)^{1/4}\\
\ll &\, q^\varepsilon (LL_2)^{1/4}P^{3/4}R^{1/2}(q+P^k|q\alpha -a|)^{1/8},
\end{align*}
and the first conclusion of Theorem \ref{theorem1.1} is therefore immediate from 
\eqref{2.2}.\par

The second conclusion of Theorem \ref{theorem1.1} requires that we remove the condition on the sum $g_\nu (\alpha ;P,R)$ that its summands come from a truncated set 
of smooth numbers. Assume the hypotheses of the statement of the theorem, and suppose that $X$ is a parameter with $P^{3/4}\le X\le P$. Then, with an obvious abuse of 
notation regarding the quantity $\mathscr L$,  the first conclusion of the theorem shows that
\[
g_2(\alpha;X,R)\ll \frac{q^\varepsilon \kp(q)^{1/2}XL(L_2\mathscr L)^{1/2}}{(1+X^k|\alpha -a/q|)^{1/2}}
+q^\varepsilon (LL_2)^{1/4}X^{3/4}R^{1/2}\bigl( q+X^k|q\alpha -a|\bigr)^{1/8}.
\]
Notice here that, as a consequence of \eqref{1.0}, one has $\mathscr L\ll (1+X^k|\alpha -a/q|)^\varepsilon $. We put $J=\lceil (\log P)/(4\log 2)\rceil$, and sum the 
contributions from $g_2(\alpha ;X,R)$ for $X=2^{-j}P$, with $0\le j<J$. Then, by applying Lemma \ref{lemma2.3}, and making use of the trivial estimate 
$g_2(\alpha;2^{-J}P,R)=O(P^{3/4})$, we find that when $k\ge 3$, one has
\[
g(\alpha;P,R)\ll P^{3/4}+\frac{q^\varepsilon \kp(q)^{1/2}PLL_2^{1/2}}{(1+P^k|\alpha -a/q|)^{1/k}}+
q^\varepsilon (LL_2)^{1/4}P^{3/4}R^{1/2}\bigl( q+P^k|q\alpha -a|\bigr)^{1/8}.
\]
When $k=2$, on the other hand, the application of Lemma \ref{lemma2.3} introduces an extra factor of $\mathscr L$ in the second summand on the right hand side of 
this upper bound, and the factor $\mathscr L^{1/2}$ remains. This confirms the second conclusion of Theorem \ref{theorem1.1}, and completes the proof of the theorem. 
\end{proof}

\section{Another enhanced major arc estimate}
The conclusion of Theorem \ref{theorem1.1} provides relatively powerful estimates for $g_\nu(\alpha;P,R)$ whenever $a\in \mathbb Z$, $q\in \mathbb N$ and 
$\alpha \in \mathbb R$ satisfy $(a,q)=1$ and the condition that $q+P^k|q\alpha -a|$ is neither too small nor larger than about $P^2$. However, when this latter quantity is 
smaller than $(\log P)^5$, these estimates are worse than trivial. Our goal in this section is to address these very small values of $q+P^k|q\alpha -a|$ via an analogue of 
\cite[Lemma 8.5]{VW1991}. Fortunately, the argument of the proof of the latter conclusion requires relatively little refinement in order that our more precise estimates be 
confirmed.\par

We begin with some auxiliary lemmata.

\begin{lemma}\label{lemma3.1}
Suppose that $X$ and $Y$ are real numbers with $0<Y<  X $, and that $f:[Y, X]\rightarrow \mathbb R$ is monotonic and differentiable on $(Y, X)$. Suppose further that $f'$ 
is continuous on $(Y,X)$, and that $f'(Y+)$ and $f'(X-)$ both exist. Then, for all real numbers $\gamma$, one has
\[
\Big| \int_{Y}^X f(\ww)e(\gamma \ww^k)\d\ww \Big| \le \bigl( |f(X-)|+|f(Y+)|\bigr) \frac{2X}{1+Y^k|\gamma|}.
\]
This upper bound holds uniformly in $Y$.
\end{lemma}

\begin{proof} The conclusion is trivial when $|\gamma|\le Y^{-k}$, so we may assume henceforth that $|\gamma|>Y^{-k}$. By integrating by parts, we find that
\begin{align*}
\int_Y^Xf(t)e(\gamma t^k)\d t=&\, \frac{f(X-)e(\gamma X^k)}{2\pi {\rm i}\gamma kX^{k-1}}-\frac{f(Y+)e(\gamma Y^k)}{2\pi {\rm i}\gamma kY^{k-1}}\\
&+\int_Y^X \left( \frac{(k-1)f(t)}{kt^k}-\frac{f'(t)}{kt^{k-1}}\right) \frac{e(\gamma t^k)}{2\pi {\rm i}\gamma}\d t ,
\end{align*}
and hence
\begin{equation}\label{3.1}
\Big|\int_Y^Xf(t)e(\gamma t^k)\d t\Big| \le \frac{|f(X-)|+|f(Y+)|}{2\pi k|\gamma|Y^{k-1}}+\frac{1}{2\pi k|\gamma|}
\int_Y^X (k-1)\frac{|f(t)|}{t^k}+\frac{|f'(t)|}{t^{k-1}}\d t.
\end{equation}
But 
\begin{align*}
\int_Y^X \frac{(k-1)|f(t)|}{t^k}\d t&\le (|f(X-)|+|f(Y+)|)\int_Y^X\frac{(k-1)}{t^k}\d t\\ & \le (|f(X-)|+|f(Y+)|)Y^{1-k}
\end{align*}
and
\[
\int_Y^X \frac{|f'(t)|}{t^{k-1}}\d t\le \Big| \int_Y^X f'(t)\d t\Big| Y^{1-k}\le (|f(X-)|+|f(Y+)|)Y^{1-k}.
\]
On substituting the last two bounds into \eqref{3.1}, the desired inequality is  readily confirmed in the case $|\gamma|>Y^{-k}$.
\end{proof}

The next lemma summarises well-known properties of the sum
\begin{equation}\label{defW}
W(q,a)=\sum_{\substack{r=1\\ (q,r)=1}}^qe(ar^k/q).
\end{equation}

\begin{lemma}\label{lemma3.0} For each natural number $k$ with $k\ge 2$, the exponential sum $W(q,a)$ satisfies the following properties.\begin{itemize}
\item[(i)] Define the integer $\theta$ via the relation $p^\theta \| k$, and suppose that $a$ is an integer with $(a,p)=1$. Then, when $p\ge 3$ and $t\ge \theta+2$, and also 
when $p=2$ and $t \ge \theta+3$, one has $W(p^t,a)=0$.
\item[(ii)] For all prime numbers $p$, one has  
\begin{equation}
\label{3.01} |W(p,a)| \le 1+ \big((k,p-1)-1\big) \sqrt{p} .
\end{equation}
\item[(iii)]When $1\le \tau< t$, one  has
\begin{equation}
\label{3.02} W(p^t,ap^\tau) = p^\tau W(p^{t-\tau},a).
\end{equation}
\item[(iv)] When $r_1$ and $r_2$ are coprime natural numbers, then for all integers $c$ one has
\begin{equation}
\label{3.03} W(r_1r_2,c)=W(r_1,cr_2^{k-1})W(r_2,cr_1^{k-1}).
\end{equation}
\end{itemize}
\end{lemma}

\begin{proof}
The vanishing property in the first clause of the lemma is \cite[Lemma 8.3]{Hua1965}. In order to establish \eqref{3.01}, we write
\[
W(p,a) =-1+\sum_{x=1}^p e(ax^k/p)
\]
and apply \cite[Lemma 4.3]{Vau1997}. The identity \eqref{3.02} follows immediately from the definition of $W(p^t,c)$. Finally, the quasi-multiplicative property \eqref{3.03} 
is established as the first clause of \cite[Lemma 8.1]{Hua1965}.
\end{proof}

In the proof of Theorem \ref{theorem1.2} we require an estimate for the sum
\begin{equation}\label{defS}
\mathscr S(q,a)=\frac{1}{q}\sum_{d\mdiv q}\psi(d)|W(d,a(q/d)^{k-1})|,
\end{equation}
in which $\psi(d)$ denotes the previously defined multiplicative function $d/\phi(d)$.

\begin{lemma}\label{lemma3S}
Let $k$ be a natural number with $k\ge 2$. Then, whenever $q\in\mathbb N$ and $a\in\mathbb Z$ are coprime, one has $\mathscr S(q,a) \le 6k \kp(q)\psi(q)$. 
\end{lemma}

\begin{proof}
We first consider the case where $q$ is a power of the prime $p$, say $q=p^l$ with $l\ge 1$. We suppose throughout that $p\nmid a$. By \eqref{defS}, one has
\begin{equation}\label{3F} 
\mathscr S(p^l,a) = \sum_{j=0}^l \frac{p^{-j}}{\phi(p^{l-j})} |W(p^{l-j}, ap^{j(k-1)})|.
\end{equation}

\par We begin by estimating the contribution to $\mathscr S(p^l,a)$ arising from the summands in \eqref{3F} with $j(k-1)\ge l-j$. In this situation, one sees from 
\eqref{defW} that  $W(p^{l-j}, ap^{j(k-1)}) = \phi(p^{l-j})$. Adopting the formulation employed already in the definition of the function $\kp$, we write $l=uk+v$ with 
$u\ge 0$ and $1\le v\le k$. Then, the integers $j$ with $j(k-1)\ge l-j$ are precisely those characterised by the condition $j>u$. The contribution from these summands on the 
right hand side of \eqref{3F} is exactly
\[
\sum_{j=u+1}^l p^{-j} \le p^{-u-1}\frac{p}{p-1}.
\]
On recalling the relation \eqref{3.02}, we therefore obtain the upper bound
\begin{equation}\label{3G}
\mathscr S(p^l,a) \le T+p^{-u-1}\psi(p^l),
\end{equation}
where
\[
T=\sum_{j=0}^u\frac{p^{-j}}{\phi(p^{l-jk})}|W(p^{l-jk}, a)|.
\]

\par We temporarily suppose that $p\nmid 6k$ and proceed to show that in this case one has
\begin{equation}\label{normal}
\mathscr S(p^l,a)\le \kp(p^l)\psi(p^l). 
\end{equation}
Indeed, if $2\le v\le k$, then for all $0\le j\le u$ we have $l-jk \ge 2$, and we may apply the first clause of Lemma \ref{lemma3.0} with $\theta=0$ to confirm that $T=0$. In 
the current context, we have $\kp(p^l)=p^{-u-1}$, so that \eqref{normal} now follows from \eqref{3G}.

This leaves the case $v=1$.   By using the same reasoning we now find that the only term that contributes to the sum $T$ is that in which $j=u$, whence
\[
T = \frac{|W(p,a)|}{p^u(p-1)}.
\]
By  \eqref{3.01} and \eqref{3G}, it therefore follows that
\[
\mathscr S(p^l,a)\le p^{-u-1}\psi(p^l)\big(2+(k-1)\sqrt{p}\big).
\]
In the current discussion, moreover, we have $p\ge 5$, and thus
\[
2+(k-1)\sqrt{p}\le k\sqrt{p}=p^{u+1}\kp(p^l).
\]
Consequently, the upper bound \eqref{normal} follows in the case $v=1$.\par

Next suppose that $p\mdiv 6k$. In this case we prove that
\begin{equation}\label{except}
\mathscr S(p^l,a)\le p\,\kp(p^l)\psi(p^l).
\end{equation}
We begin with the case where $p\ge 3$. When $p\mdiv k$, we find that Lemma \ref{lemma3.0} applies with $\theta\ge 1$ and shows that whenever $W(p^{l-jk},a)$ is 
non-zero, then $l-jk\le \theta+1$. This implies that $(u-j)k\le \theta +1-v\le \theta$. However, the upper bound $p^\theta\le k$ implies that $\theta<k$. Hence, it is again 
the case that the only summand that contributes to $T$ is that with $j=u$, and then a trivial bound for $|W(p^v,a)|$ leads to the relation
\[
T=p^{-u}\frac{|W(p^v,a)|}{\phi(p^v)}\le p^{-u}=(p-1)p^{-u-1}\psi(p^l).
\]
We therefore deduce from \eqref{3G} and the definition of $\kp(p^l)$ that
\begin{equation}\label{S}
\mathscr S(p^l,a)\le p^{-u}\psi(p^l)\le p\kp(p^l)\psi(p^l).
\end{equation}
If $p=3$ but $3\nmid k$, meanwhile, then $\theta=0$ and the above argument still applies with some obvious modification to deliver \eqref{S}. Thus, the bound 
\eqref{except} holds whenever $p\ge 3$.\par  
 
This leaves the prime $p=2$ for consideration. We first consider the situation in which $k\ge 3$. When $k$ is odd, we have $\theta=0$. When $k$ is even and $k\ge 4$, 
meanwhile, we have $\theta\ge 1$ and $2^\theta\le k$.  It follows that in both scenarios we have $\theta\le k-2$. Here, when $0\le j\le u$ and $W(2^{l-jk}, a)$ is non-zero, it 
follows from Lemma \ref{lemma3.0} that $l-jk\le \theta +2\le k$. This implies that $(u-j)k\le k-v\le k-1$. Thus, we again see that the only summand that contributes to $T$ is 
that with $j=u$, and the proof of \eqref{except} runs as before via \eqref{S}.\par

The final case left is that where $p=k=2$. We now have $\theta=1$, and Lemma 3.2 shows that when $1\le j \le u$, one has $W(2^{l-2j},a)=0$ unless $l-2j\le 3$. Thus, we 
may restrict the range of summation in $T$ to those values of $j$ with $2(u-j)\le 3-v$. In case we have $v=2$ or $u=0$, we conclude that $j=u$ and may then one last time 
appeal to the foregoing argument to verify \eqref{except}. We are then reduced to the situation where $v=1$ and $u\ge 1$. This case is different, for now the summands 
with $j=u$ and $j=u-1$ make non-zero contributions to $T$, and we find that
\[
T=2^{-u}\frac{|W(2,a)|}{\phi(2)}+2^{1-u}\frac{|W(8,a)|}{\phi(8)}\le 3\cdot 2^{-u}=3\cdot 2^{-u-1}\psi(2^l).
\]
By \eqref{3G}, one now finds that $\mathscr S(2^l, a)\le 2^{1-u}\psi(2^l)$. Meanwhile, since $l$ is odd, one has $\kappa(2^l)= 2^{-u+1/2}>2^{-u}$, and \eqref{except} 
follows.\par
 
With \eqref{normal} and \eqref{except} in hand, the proof of the lemma is swiftly completed. The convolution \eqref{defS} passes the quasi-multiplicative property 
\eqref{3.03} on to $\mathscr S(q,a)$, delivering the formula
\[ 
\mathscr S(q_1q_2,a)=\mathscr S(q_1,aq_2^{k-1})\mathscr S(q_2,aq_1^{k-1})
\]
that is valid for all coprime natural numbers $q_1$, $q_2$ and all integers $a$. When $(q,a)=1$, one may repeatedly apply this identity to reduce to the exact prime powers 
dividing $q$.  Then, by \eqref{normal} and \eqref{except}, one finds that
\[
\mathscr S(q,a) \le \kp(q)\psi(q)\prod_{p|6k}p.
\]
The conclusion of the lemma now follows with a trivial estimate for the final product.
\end{proof} 
 
\begin{proof}[The proof of Theorem 1.2]
We are now fully equipped to embark on the proof of Theorem \ref{theorem1.2}. Suppose then that the hypotheses of this theorem are satisfied. We begin with the special 
case where
\[
2\le R\le \exp\big((\log P)^{1/3}\big).
\] 
Here, a straightforward estimation is enough. In fact, it follows easily from \cite[Chapter~III.5, Theorem 1]{Ten} that, uniformly for $2\le R\le P$, one has
\[
\text{card}(\mathscr A(P,R))\ll P\exp\Big(- \frac{\log P}{2\log R}\Big).
\]
In the range for $R$ that is currently under consideration, it therefore suffices to remark that by means of a trivial estimate one has
\[
|g(\alpha)|+|g_\nu(\alpha)|\ll \text{card}(\mathscr A(P,R))\ll P\exp\big(-\tfrac 12 (\log P)^{2/3}\big).
\]
This upper bound is stronger than those claimed in Theorem \ref{theorem1.2}.\par

From now onwards, we shall suppose that
\begin{equation}\label{RR}
\exp\big((\log P)^{1/3}\big)\le R\le P^{2/3}, 
\end{equation}
and proceed to establish the estimate for $g_\nu(\alpha)$ contained in Theorem \ref{theorem1.2}. A considerable part of our argument is largely identical with the 
demonstration of \cite[Lemma 8.5]{VW1991}. We therefore opt for an exposition where the reader is expected to be familiar with the latter source, and we apply the notation 
of \cite{VW1991} without reintroducing it here.\par

In \cite{VW1991} it is assumed that $P$ and $R$ are linked via the equation $R=P^\eta$, for some fixed $\eta\in(0,1/2)$. Here, we read this relation as a definition for $\eta$ 
and check the argument given in \cite[pp. 56--58]{VW1991} for uniformity as $R$ ranges over the interval \eqref{RR}. We begin by noting that
\begin{equation}\label{Dyadic}
g_\nu(\alpha;P,R)=g(\alpha;P,R)-g(\alpha;P/\nu,R)
\end{equation}
and apply the  initial decomposition of $g(\alpha)$ leading up to \cite[equation (8.3)]{VW1991}. On subtracting the resulting expressions for $g(\alpha;P,R)$ and 
$g(\alpha;P/\nu,R)$, one arrives at a version of \cite[equation (8.3)]{VW1991} for $g_\nu(\alpha;P,R)$ in which the expression $\mathscr M_d$ in the main term therein is 
now replaced by the difference of the corresponding terms in \cite{VW1991}, say $\mathscr M_d(Q)$ and $\mathscr M_d(Q/\nu)$. If one follows through the treatment in 
\cite{VW1991} and invokes Lemma \ref{lemma3.1} whenever our main source calls upon \cite[Lemma 8.2]{VW1991}, then factors $(1+Q^k|\gamma|)^{-1/k}$ that appear in 
the original estimations are now replaced by $(1+Q^k|\gamma|)^{-1}$. A careful analysis reveals that all of these new estimates are valid uniformly for $R$ ranging over the 
interval described in \eqref{RR}, with one exception: we obtain the new bound $\mathscr N_4(v) \ll Q(1+Q^k|\gamma|)^{-1}$ by the principle just described, but this leads 
to a satisfactory bound for $\mathscr N_3$ only when $\log Q/\log R$ remains bounded. Note that  $Q\gg P^{1-\varepsilon}$, whence  the treatment in \cite{VW1991} 
requires $\eta$ to vary over a compact subset of $(0,1/2)$ at most. We note that this is the only instance where the uniformity that we now desire is not covered by the work 
in \cite{VW1991}. However, an inspection of the treatment of $\mathscr N_3$ in the latter source reveals that, should the estimate
\begin{equation} \label{N3}
\int_0^\infty N_d(R^v) |\mathscr N_4(v)|\d v\ll Q(1+Q^k|\gamma|)^{-1}
\end{equation}  
be established, then one would arrive at the bound
\[
g_\nu(\alpha ;P,R)\ll \frac{P\mathscr S(q,a)}{1+P^k|\beta|}+P(1+P^k|\beta|)\exp \bigl( -c(\log P)^{1/2}\bigr)
\]
that is tantamount to the conclusion in the penultimate paragraph on \cite[page 58]{VW1991}. The first conclusion of Theorem \ref{theorem1.2} would then follow from 
Lemma \ref{lemma3S}. 

The task remaining to us is to establish \eqref{N3}. In preparation for this, we recall the differencing process induced by \eqref{Dyadic}, noting that this impacts the 
expression $\mathscr N_3$ accordingly. The integral that is  $J(Q,R)$ in \cite{VW1991} now becomes the function of $v\in\mathbb R$ defined by
\[
\int_{Q/\nu}^Q 2\pi\mathrm i\gamma kX^{k-1}e(\gamma X^k)X\rho_1\Big(\frac{\log X}{\log R}-v\Big)\d X.
\]   
Here, the function $\rho_1$ is that defined in the prelude to \cite[equation (8.4)]{VW1991}. Since $\rho_1(t)=0$ for $t\le 1$, we see that the above integral is zero unless 
$Q\ge R^{v+1}$, and in the latter case the integral may be rewritten as
\[
\int_{\Theta(v)}^Q  2\pi\mathrm i\gamma kX^{k-1}e(\gamma X^k)X\rho_1\Big(\frac{\log X}{\log R}-v\Big)\d X,
\] 
where $\Theta(v)=\max \{ R^{v+1},Q/\nu\}$. Integrating by parts, one encounters the analogue of the original function $\mathscr N_4(v)$ that in our new context is defined 
by
\[
\mathscr N_4(v)= \int_{\Theta(v)}^Q e(\gamma X^k)\bigg( \rho_1\Big(\frac{\log X}{\log R}-v\Big)+(\log R)^{-1}\rho_2\Big(\frac{\log X}{\log R}-v\Big)\bigg) \d X.
\]
Again, the function $\rho_2$ is that defined in \cite{VW1991}. The functions $-\rho_1$ and $\rho_2$ are positive and decreasing, with continuous derivatives, on the range of 
integration, and thus we may apply Lemma \ref{lemma3.1}. For $j=1$ and 2, this gives
\[
\int_{\Theta(v)}^Q e(\gamma X^k)\rho_j\Big(\frac{\log X}{\log R}-v\Big)\d X \ll Q(1+Q^k|\gamma|)^{-1}\Big|\rho_j\Big(\frac{\log \Theta(v)}{\log R}-v\Big)\Big| .
\] 

\par From the definition of $\mathscr N_4(v)$ and properties of $\rho_1$ and $\rho_2$, we see that $\mathscr N_4(v)$ vanishes for $v>\log (Q/R)/\log R$. We therefore 
deduce that
\[
\int_0^\infty N_d(R^v) |\mathscr N_4(v)|\d v \le \mathscr IQ(1+Q^k|\gamma|)^{-1},
\]
where
\[
\mathscr I=\int_0^{\log (Q/R)/\log R}\bigg(  \Big|\rho_1\Big(\frac{\log \Theta(v)}{\log R}-v\Big)\Big| +(\log R)^{-1}\rho_2\Big( \frac{\log \Theta(v)}{\log R}-v\Big)\bigg) 
\d v.
\] 
We have $\Theta(v)=R^{v+1}$ whenever $v+1>\log (Q/\nu )/\log R$, and otherwise $\Theta(v)=Q/\nu$. Now,  the range of integration in the integral $\mathscr I$ is 
$0\le v \le  \log (Q/R)/\log R$. Thus, the part where  $\Theta(v)=R^{v+1}$ holds is an interval of length $\log \nu /\log R$, and in the integrand we then have 
$-\rho_1(1+)=1$ and $\rho_2(1+)=1$. This part therefore contributes a bounded amount.  The contribution of the remaining range is bounded by
\[
\int_0^\infty ( |\rho_1(t)|+|\rho_2(t)|)\d t,
\]
and this integral exists owing to the exponential decay of Dickman's function that is a majorant for both $\rho_1$ and $\rho_2$. We conclude that $\mathscr I\ll 1$, and 
hence that the upper bound \eqref{N3} is valid. The proof of the first conclusion in Theorem \ref{theorem1.2} is now complete, and indeed we have proved it uniformly in 
$2\le R\le P^{2/3}$.\par 

We now prove the second conclusion of Theorem \ref{theorem1.2} by making use of a dyadic dissection. Let $J$ be the smallest natural number having the property that 
$2^J>P^{1/9}$. Then
\begin{equation}\label{3.15}
g(\alpha;P,R)=\sum_{j=0}^J g_2(\alpha;2^{-j}P,R)+O(P^{8/9}).
\end{equation}
Fix a number $A>1$ and apply the first conclusion of Theorem \ref{theorem1.2} that is already established, but with $2A$ in place of $A$, and with $2^{-j}P$ in place of $P$, 
for $0\le j\le J$. Notice that we have established this consequence of Theorem \ref{theorem1.2} in the range $2\le R\le P^{2/3}$ that is wider than the corresponding range 
$2\le R\le P^{1/2}$ claimed in the statement of the theorem. Since one has $(\log X)^{2A}>(\log P)^A$ and $X^{2/3}>P^{1/2}$ for $P^{8/9}\le X\le P$, the desired bound 
for $g(\alpha;P,R)$  follows via the dyadic decomposition \eqref{3.15} and an application of Lemma \ref{lemma2.3}.
\end{proof}

\section{An auxiliary pruning lemma}
We now discuss a pruning lemma that should be of wider utility than the immediate applications of this paper demand. For this reason, we keep the presentation in this 
section fairly self-contained. As usual, we fix a natural number $k$, which in this section is assumed to satisfy $k\ge 3$, and define the function $\kp(q)$ as in the preamble 
to Theorem \ref{theorem1.1}. Now choose parameters $\delta$, $Q$, $X$, and $Y$ with
\begin{equation} \label{para}
\delta>1,\quad Q\ge 1,\quad X\ge 1,\quad Y> 0,\quad QY \le \tfrac14 X^k.
\end{equation}
When $0\le a\le q\le Q$ and $(a,q)=1$, the intervals defined by
\[
\mathscr M(q,a) = \{\alpha\in[0,1): |q\alpha -a|\le YX^{-k}\}
\]
are disjoint. We denote the union of these intervals by $\mathscr M = \mathscr M_{Q,X,Y}$ and define the function $\Ups: \mathscr M\to [0,1]$ by putting
\[
\Ups(\alpha)=\kp(q)^2(1+X^k|\alpha -a/q|)^{-\delta} 
\]
when $\alpha \in \mathscr M(q,a)$.

\begin{lemma}\label{lemma4.1}
Let $k$ and $t$ be natural numbers with $k\ge 3$ and $t\ge \lfloor k/2\rfloor$. Suppose that the numbers $\delta$, $Q$, $X$ and $Y$ satisfy \eqref{para}. Then, for any 
subset $\mathscr Z$ of $[1,X]\cap \mathbb Z$, one has for each $\varepsilon >0$ the estimate
\[
\int_{\mathscr M}\Ups(\alpha) \biggl| \sum_{x\in \mathscr Z}e(\alpha x^k)\biggr|^{2t}\d\alpha \ll Q^\varepsilon X^{2t-k}+X^{-k}Q^{t+1+\varepsilon}.
\]
\end{lemma}

\begin{proof} Our treatment is a development of the argument of the proof of \cite[Lemma 5.4]{VW2000}. We have
\begin{equation}\label{4.1}
\int_{\mathscr M}\Ups(\alpha) \biggl| \sum_{x\in \mathscr Z}e(\alpha x^k)\biggr|^{2t}\d\alpha \le \sum_{1\le q\le Q}\kp(q)^2\int_{-Y/X^k}^{Y/X^k} 
\frac{W(\beta,q,\mathscr Z)}{(1+X^k|\beta|)^\delta}\d\beta ,
\end{equation}
in which we write
\[
W(\beta,q,\mathscr Z)=\sum_{a=1}^q\biggl| \sum_{x\in \mathscr Z}e(x^k(\beta +a/q))\biggr|^{2t}.
\]
Write
\[
\psi(\bfx)=\sum_{i=1}^t(x_{2i-1}^k-x_{2i}^k).
\]
Then we find by orthogonality that
\begin{align}
W(\beta ,q,\mathscr Z)&=q\sum_{\substack{\bfx \in \mathscr Z^{2t}\\ q\mdiv \psi(\bfx)}}e(\beta \psi(\bfx)) \notag \\
&\le q\sum_{\substack{1\le x_1,\ldots ,x_{2t}\le X\\ q\mdiv \psi(\bfx)}}1\notag \\
&=\sum_{a=1}^q\biggl| \sum_{1\le x\le X}e(x^ka/q)\biggr|^{2t}.\label{4.2}
\end{align}

\par We now enhance the corresponding argument of \cite[Lemma 5.4]{VW2000} by applying \cite[Theorem 4.1]{Vau1997} to see that when $(a,q)=1$, one has
\[
\sum_{1\le x\le X}e(x^ka/q)=Xq^{-1}S(q,a)+O(q^{1/2+\varepsilon}),
\]
where
\[
S(q,a)=\sum_{r=1}^q e(ar^k/q).
\] 
By \cite[Lemmata 4.3, 4.4 and 4.5]{Vau1997}, we therefore deduce in such circumstances that
\[
\sum_{1\le x\le X}e(x^ka/q)\ll X\kp(q)+q^{1/2+\eps}.
\]
Removing now the condition $(a,q)=1$, we find that in general one has the upper bound
\begin{equation}\label{4.3}
\sum_{1\le x\le X}e(x^ka/q)\ll X\kp(q/(q,a))+(q/(q,a))^{1/2+\eps}.
\end{equation}

\par Next, we substitute \eqref{4.3} into \eqref{4.2} to obtain the estimate
\begin{align}
W(\beta, q,\mathscr Z)&\ll X^{2t}\sum_{a=1}^q\kp(q/(q,a))^{2t}+\sum_{a=1}^q(q/(q,a))^{t+\varepsilon}\notag \\
&=X^{2t}\sum_{r\mdiv q}r\kp(r)^{2t}+\sum_{r|q}r^{t+1+\varepsilon}.\label{4.4}
\end{align}
Write
\[
\sigma (q)=\sum_{r\mdiv q}r\kp(r)^{2t}.
\]
Then we find that $\sigma(q)$ is the subject of the discussion of \cite{VW2000} leading from equation (5.9) to the conclusion of the proof of Lemma 5.4 of the latter source. 
Unfortunately, the estimations there are subject to the condition $k\ge 4$, but subject to this constraint it is shown that
\begin{equation} \label{VWbound}
\sum_{1\le q\le Q}\kp(q)^2\sum_{r\mdiv q}r\kp(r)^{2t}\ll Q^\varepsilon .
\end{equation}
This bound remains valid for $k=3$, too, as we shall demonstrate momentarily. Since one also has the bound
\[
\sum_{r|q}r^{t+1+\varepsilon}\ll q^{t+1+2\varepsilon},
\]
we see from the definition of $\kp(q)$ that there is a positive number $B$ having the property that 
\[
\sum_{1\le q\le Q}\kp(q)^2\sum_{r|q}r^{t+1+\varepsilon}\ll Q^{t+1+2\varepsilon}\prod_{p\le Q}(1+Bp^{-1})\ll Q^{t+1+3\varepsilon}.
\]
Thus, we deduce from \eqref{4.1} and \eqref{4.4} that
\begin{align*}
\int_{\mathscr M}\Ups(\alpha)\biggl| \sum_{x\in \mathscr Z}e(\alpha x^k)\biggr|^{2t}\d\alpha \ll &\, X^{2t-k}\sum_{1\le q\le Q}\kp(q)^2\sum_{r\mdiv q}r\kp(r)^{2t}\\
&\,+X^{-k}\sum_{1\le q\le Q}\kp(q)^2\sum_{r|q}r^{t+1+\varepsilon},
\end{align*}
and the conclusion of the lemma is now immediate.\par

It remains to confirm \eqref{VWbound} when $k=3$. In view of the crude estimate $\kp_3(q)\ll q^{-1/3}$  it suffices to treat the case $t=1$. In this case, the crude bound 
shows that $\sigma(p^l)\ll l p^{l/3}$, whence $\kp(p^l)^2\sigma(p^l) \ll lp^{-l/3}$. We therefore deduce that
\[
\sum_{l=4}^\infty \kp(p^l)^2\sigma(p^l) \ll p^{-4/3}.
\]
Meanwhile, working directly from the definition of $\kp(q)$, we find that there is a constant $C>1$ with the property that
\[
\kp(p)^2\sigma(p) + \kp(p^2)^2\sigma(p^2)+ \kp(p^3)^2\sigma(p^3) \le Cp^{-1}.
\]
The bound \eqref{VWbound} now follows from the estimate
\[
\sum_{1\le q\le Q}\kp(q)^2\sigma(q) \le \prod_{p\le Q} \sum_{l=0}^\infty \kp(p^l)^2\sigma(p^l) \le \prod_{p\le Q} \big(1+Cp^{-1}+O(p^{-4/3})\big).
\]
\end{proof}

\section{Mean values on sets of major arcs}
Our goal in this section is the proof of Theorem \ref{theorem1.3} concerning upper bounds for moments of the smooth Weyl sum $g(\alpha;P,R)$ on sets of major arcs.

\begin{lemma}\label{lemma5.1}
Let $k$ and $t$ be natural numbers with $k\ge 3$ and $t\ge \lfloor k/2\rfloor$, and let $\nu$ be a real number with $\nu>1$. Suppose that $R$ and $P$ are real numbers 
with $2\le R\le P^\eta$, where $\eta$ is a positive number sufficiently small in terms of $k$ and $\varepsilon$. Finally, let $\omega$ be a positive number with
\begin{equation}\label{5.1}
\omega <\frac{2t+4}{t+10},
\end{equation}
and suppose that $1\le Q\le P^\omega$. Then, there exists a positive number $\tau$ with the property that, whenever $1\le X\le P$, one has
\[
\int_{\grM(Q)}|g_\nu(\alpha ;X,R)|^{2t+4}\d\alpha \ll Q^\varepsilon X^{2t+4-k}(\log X)^{4+\varepsilon}+P^{2t+4-k-\tau}. 
\]
\end{lemma}

\begin{proof} We begin by observing that the conclusion of the lemma is trivial when $X\le P^{1/2}$, so we suppose henceforth that $P^{1/2}<X\le P$.  Consider a real 
number $\alpha$ with $\alpha\in \grM(Q)$. Then there exists $a\in \mathbb Z$ and $q\in \mathbb N$ with $1\le q\le Q$ and $(a,q)=1$ with $|q\alpha -a|\le QP^{-k}$. 
Consequently, it follows from Theorem \ref{theorem1.1} that
\[
g_\nu(\alpha ;X,R)\ll \frac{q^\varepsilon \kp(q)^{1/2}X(\log X)^{1+\varepsilon}}{(1+X^k|\alpha -a/q|)^{1/2-\varepsilon}}+
X^{3/4}R^{1/2}(\log X)^{1/4+\varepsilon}Q^{1/8+\varepsilon}. 
\]

\par Next, we prepare for an application of Lemma  \ref{lemma4.1} where we work with $X$ and $Q$ as above, and we choose $Y= Q(X/P)^k$ to arrange that 
$\mathfrak M(Q)=\mathscr M_{Q,X,Y}$. We define the function $\Ups(\alpha)$ on $\mathscr M_{Q,X,Y}$ as in the preamble of Lemma \ref{lemma4.1}, with $\delta=3/2$, 
and then recast the preceding bound for $g_\nu(\alpha ;X,R)$ in the form
\[
g_\nu(\alpha ;X,R)\ll Q^\varepsilon X(\log X)^{1+\varepsilon}\Ups(\alpha)^{1/4}+X^{3/4}R^{1/2}(\log X)^{1/4+\varepsilon}Q^{1/8+\varepsilon}.
\]
Thus, we deduce that the mean value
\begin{equation}\label{5.2}
I_0=\int_{\grM(Q)}|g_\nu(\alpha ;X,R)|^{2t+4}\d\alpha 
\end{equation}
satisfies the bound
\begin{equation}\label{5.3}
I_0\ll I_1+I_2,
\end{equation}
where
\begin{equation}\label{5.4}
I_1=Q^\varepsilon \bigl( X(\log X)^{1+\varepsilon}\bigr)^4\int_{\grM(Q)}\Ups(\alpha)|g_\nu(\alpha;X,R)|^{2t}\d\alpha
\end{equation}
and
\begin{equation}\label{5.5}
I_2=Q^{1/2+\varepsilon}X^{3+\varepsilon}R^2\int_{\grM(Q)}|g_\nu(\alpha ;X,R)|^{2t}
\d\alpha .
\end{equation}

\par In order to bound the integral occurring in \eqref{5.5}, we recall \eqref{5.2} and apply H\"older's inequality to obtain the estimate
\begin{align*}
\int_{\grM(Q)}|g_\nu(\alpha;X,R)|^{2t}\d\alpha &\ll \Bigl( \int_{\grM(Q)}\d\alpha \Bigr)^{2/(t+2)}\Bigl( \int_{\grM(Q)}|g_\nu(\alpha;X,R)|^{2t+4}\d\alpha \Bigr)^{t/(t+2)}\\
&\ll (Q^2P^{-k})^{2/(t+2)}I_0^{t/(t+2)}.
\end{align*} 
Applying Lemma \ref{lemma4.1} to estimate the mean value $I_1$ defined in \eqref{5.4}, we therefore deduce from \eqref{5.3} that
\[
I_0\ll Q^\varepsilon X^4(\log X)^{4+\varepsilon}(X^{2t-k}+X^{-k}Q^{t+1})+Q^{1/2+\varepsilon}X^{3+\varepsilon}R^2(Q^2P^{-k})^{2/(t+2)}I_0^{t/(t+2)},
\]
whence
\begin{equation}\label{5.6}
I_0\ll Q^\varepsilon X^4(\log X)^{4+\varepsilon}(X^{2t-k}+X^{-k}Q^{t+1})+Q^2P^{-k}\bigl( Q^{1/8+\varepsilon}X^{3/4+\varepsilon}R^{1/2}\bigr)^{2t+4}.
\end{equation}

\par Since we have $Q\le P^\omega$ with $\omega$ satisfying the bound \eqref{5.1}, we find that there is a positive number $\tau$ for which
\[
Q^2P^{-k}(Q^{1/8}X^{3/4})^{2t+4}=Q^{(t+10)/4}X^{(3t+6)/2}P^{-k}\ll P^{(t+2)/2-2\tau}X^{(3t+6)/2}P^{-k}.
\]
But $X\le P$, and thus
\begin{equation}\label{5.7}
Q^2P^{-k}\bigl( Q^{1/8+\varepsilon}X^{3/4+\varepsilon}R^{1/2}\bigr)^{2t+4}\ll P^{2t+4-k-\tau}.
\end{equation}

\par The upper bound \eqref{5.7} provides satisfactory control over the final term on the right hand side of \eqref{5.6}. We turn next to consider the second half of the first 
term on the right-hand side of this relation. Our goal is to show that $X^{4-k}Q^{t+1}<P^{2t+4-k-2\tau}$. When $k=3$ this bound is easily obtained by employing the 
upper bounds $X\le P$ and $Q\le P^\omega$. One has
\[
\frac{(t+1)(2t+4)}{t+10}=2t-\frac{14t-4}{t+10}<2t-2\tau,
\]
so that in view of \eqref{5.1}, one has
\[
XQ^{t+1}<P^{2t+1-2\tau},
\]
and the desired conclusion follows. We may henceforth suppose, therefore, that $k\ge 4$. Here, we begin by observing that a trivial estimate for $g_\nu(\alpha;X,R)$ shows 
that when
\[
X^{2t+4}Q^2P^{-k}<P^{2t+4-k-\tau},
\]
then
\[
I_0= \int_{\grM(Q)}|g_\nu(\alpha ;X,R)|^{2t+4}\d\alpha \ll X^{2t+4}\text{mes}(\grM(Q))\ll P^{2t+4-k-\tau}.
\]
We may therefore suppose that $X^{2t+4}Q^2\ge P^{2t+4-\tau}$, whence
\[
X\ge P^{1-2/(t+10)}Q^{-1/(t+2)}.
\]
But then, on recalling \eqref{5.1} and the hypothesis $Q\le P^\omega$, we find that
\[
X^{4-k}Q^{t+1}\le P^{4-k+2(k-4)/(t+10)}Q^{t+1+(k-4)/(t+2)}\le P^{4-k+2\Lambda/(t+10)},
\]
where $\Lambda=(t+1)(t+2)+2(k-4)$. However, since $k\le 2t+1$, one finds that
\[
\Lambda=t(t+10)-(7t-2k+6)\le t(t+10)-(3t+4)<(t+10)(t-\tau).
\]
We therefore conclude that when $k\ge 4$, we may assume that $X^{4-k}Q^{t+1}<P^{2t+4-k-2\tau}$. Then in all cases, we find that the bound \eqref{5.6} implies via 
\eqref{5.7} that
\[
I_0\ll Q^\varepsilon X^{2t+4-k}(\log X)^{4+\varepsilon}+P^{2t+4-k-\tau},
\]
and thus the proof of the lemma is complete.
\end{proof}

Before announcing a corollary of this conclusion, we recall the definition \eqref{1.3} of $\grN(Q)$.

\begin{corollary}\label{corollary5.2}
Assume the hypotheses of the statement of Lemma \ref{lemma5.1}, and let $\tau$ be the positive number supplied by its conclusion. Also, let $\omega'$ be a positive 
number with
\begin{equation}\label{omeg0}
\omega'<\frac{2k}{k+4},
\end{equation}
and put $\Omega=\min\{ \omega,\omega'\}$. Then provided that $1\le Q\le P^\omega$, one has
\[
\int_{\grM(Q)}|g(\alpha;P,R)|^{2t+4}\d\alpha \ll Q^\varepsilon P^{2t+4-k}(\log P)^{2t+8+\varepsilon}.
\]
Moreover, when $u>2t+4$ and $A$ is sufficiently large in terms of $\tau$ and $u$, then whenever $(\log P)^A\le Q\le P^\Omega$, one has
\[
\int_{\grN(Q)}|g(\alpha;P,R)|^u\d\alpha \ll P^{u-k}Q^{\varepsilon -(u-2t-4)/(2k)}.
\]
\end{corollary}

\begin{proof} Write $\zeta=2/(t+10)$. Then by dividing the summation underlying $g(\alpha ;P,R)$ into dyadic intervals, we find that
\[
|g(\alpha;P,R)|\le \sum_{\substack{j=0\\ 2^j\le P^\zeta }}^\infty |g_2(\alpha;2^{-j}P,R)|+O(P^{1-\zeta}).
\]
By H\"older's inequality, therefore, we have
\[
|g(\alpha ;P,R)|^{2t+4}\ll (\log P)^{2t+3}\sum_{\substack{j=0\\ 2^j\le P^\zeta }}^\infty |g_2(\alpha ;2^{-j}P,R)|^{2t+4}+P^{(2t+4)(1-\zeta)}.
\]
Put
\[
I=\int_{\grM(Q)}|g(\alpha;P,R)|^{2t+4}\d\alpha .
\]
Then it follows from Lemma \ref{lemma5.1} that there is a real number $X$ satisfying $P^{1-\zeta}\le X\le P$ for which
\begin{align*}
I&\ll (\log P)^{2t+4}\int_{\grM(Q)}|g_2(\alpha ;X,R)|^{2t+4}\d\alpha +P^{(2t+4)(1-\zeta)}Q^2P^{-k}\\
&\ll Q^\varepsilon P^{2t+4-k}(\log P)^{2t+8+\varepsilon}+P^{2t+4-k-\tau/2}+Q^2P^{(2t+4)(1-\zeta)-k}.
\end{align*}
However, our hypotheses concerning $Q$ and $\tau$ ensure that $Q^2<P^{(2t+4)\zeta-\tau}$, and thus we conclude that
\[
\int_{\grM(Q)}|g(\alpha;P,R)|^{2t+4}\d\alpha \ll Q^\varepsilon P^{2t+4-k}(\log P)^{2t+8+\varepsilon}+P^{2t+4-k-\tau/2}.
\]
This delivers the first mean value estimate of the corollary.\par

For the second bound, we consider a typical point $\alpha \in \grN(Q)$. There exists $a\in \mathbb Z$ and $q\in \mathbb N$ with $(a,q)=1$ and
\[
Q/2\le q+P^k|q\alpha -a|\le 2Q.
\]
Then we see from Theorem \ref{theorem1.1} that
\[
g(\alpha;P,R)\ll P(\log P)^{1+\varepsilon}Q^{\varepsilon-1/(2k)}+P^{1+\varepsilon}(Q/P^2)^{1/8}.
\]
However, we may suppose that $Q\le P^{\omega'}$, and so it follows from \eqref{omeg0} that there is a positive number $\xi$ having the property that 
$Q\le P^{2k/(k+4)-16\xi}$. Thus, we have
\[
(Q/P^2)^{1/8}\le P^{-1/(k+4)-2\xi}\le P^{-2\xi}Q^{-1/(2k)}.
\]
Since $Q\ge (\log P)^A$ with $A$ sufficiently large, it follows that
\[
g(\alpha;P,R)\ll PQ^{2\varepsilon-1/(2k)}.
\]
By making use of this bound in combination with the first conclusion of the corollary, we arrive at the upper bound
\begin{align*}
\int_{\grN(Q)}|g(\alpha;P,R)|^u\d\alpha &\ll (PQ^{\varepsilon-1/(2k)})^{u-2t-4}\int_{\grM(Q)}|g(\alpha ;P,R)|^{2t+4}\d\alpha \\
&\ll (PQ^{\varepsilon -1/(2k)})^{u-2t-4}Q^\varepsilon P^{2t+4-k}(\log P)^{2t+8+\varepsilon}.
\end{align*}
Again making use of the assumption $Q\ge (\log P)^A$ with $A$ sufficiently large, we see that $(\log P)^{2t+9}\ll Q^\varepsilon$, and hence
\[
\int_{\grN(Q)}|g(\alpha ;P,R)|^u\d\alpha \ll P^{u-k}Q^{u\varepsilon-(u-2t-4)/(2k)}.
\]
This completes the proof of the corollary.
\end{proof}

Corollary \ref{corollary5.2} is of strength sufficient to deliver the conclusions of Theorem \ref{theorem1.3} in all cases with $Q\ge (\log P)^A$, for a suitably large positive 
number $A$. We now attend to the remaining values of $Q$ with $1\le Q<(\log P)^A$.

\begin{lemma}\label{lemma5.3}
Let $k$ be a natural number with $k\ge 4$, and let $A$ be a positive real number. Then, whenever $Q$, $R$ and $P$ are real numbers with $2\le R\le P^{1/2}$ and 
$1\le Q\le (\log P)^A$, one has  
\begin{equation}\label{kplus1}
\int_{\grM(Q)}|g(\alpha ;P,R)|^{k+1}\d\alpha \ll  P \log Q.
\end{equation}
Moreover, when $u>k+1$, one also has
\begin{equation}\label{uuu}
\int_{\grN(Q)}|g(\alpha ;P,R)|^u\d\alpha \ll P^{u-k}Q^{\varepsilon -(u-k-1)/k}.
\end{equation}
In the case $k=3$, the bound \eqref{uuu} remains valid, while in \eqref{kplus1} the upper bound for the integral has to be replaced by $P(\log Q)^{82}$.
\end{lemma}

\begin{proof} Suppose that $1\le Q\le (\log P)^A$ and $\alpha \in \grM(Q)$. Then it follows from Theorem \ref{theorem1.2} that when $a\in \mathbb Z$ and 
$q\in \mathbb N$ satisfy $(a,q)=1$, with $q\le Q$ and $|q\alpha -a|\le QP^{-k}$, one has
\begin{align}
g(\alpha ;P,R)&\ll \frac{P\kp(q)\psi(q)}{(1+P^k|\alpha -a/q|)^{1/k}}+P\exp \bigl( -c(\log P)^{1/2}\bigr) Q\notag \\
&\ll \frac{ P\kp(q)\psi(q)}{(1+P^k|\alpha -a/q|)^{1/k}}.\label{5.8}
\end{align}
Consequently, one finds that
\begin{align}\label{euler}
\int_{\grM(Q)}|g(\alpha ;P,R)|^{k+1}\d\alpha &\ll P^{k+1}\sum_{1\le q\le Q}\sum_{\substack{a=1\\ (a,q)=1}}^q \big(\kp(q)\psi(q)\big)^{k+1}
\int_{-\infty}^\infty \frac{\d\beta}{(1+P^k|\beta|)^{1+1/k}}\notag\\
&\ll  P\sum_{1\le q\le Q}\phi(q)\big(\kp(q)\psi(q)\big)^{k+1}\notag\\
&\ll P\prod_{p\le Q}\Big(1+ \psi(p)^k\sum_{l=1}^\infty p^l\kp(p^l)^{k+1}\Big).
\end{align}
The crude bound $\kp(p^l)\le p^{-l/k}$ is applied for $l>k$. Then, inspecting the definition of the function $\kp_k(p^l)$, we find that when $k\ge 4$ one has
\[
\sum_{l=1}^\infty p^l\kp(p^l)^{k+1}\le \frac{1}{p}+O(p^{-1-1/k}),
\]
whereas in the case $k=3$ the above sum can only be bounded above by $82p^{-1}+O(p^{-4/3})$. But then, for $k\ge 4$, we also have
\[
1+\psi(p)^k\sum_{l=1}^\infty p^l\kp(p^l)^{k+1}\le 1+\frac{1}{p}+O(p^{-1-1/k}),
\]
with an obvious adjustment in the case $k=3$. The first claim of the lemma is now immediate, after inserting these final estimates into \eqref{euler}.\par

When $u>k+1$, we observe as in the proof of our previous corollary that when $\alpha \in \grN(Q)$, then the upper bound \eqref{5.8} yields
\[
g(\alpha ;P,R)\ll PQ^{\varepsilon -1/k}.
\]
Thus, we deduce from \eqref{kplus1} that when $u>k+1$, one has
\begin{align*}
\int_{\grN(Q)}|g(\alpha ;P,R)|^u\d\alpha &\ll (PQ^{\varepsilon -1/k})^{u-k-1}\int_{\grM(Q)} |g(\alpha ;P,R)|^{k+1}\d\alpha \\
&\ll (PQ^{\varepsilon -1/k})^{u-k-1}Q^\varepsilon P\\
&\ll P^{u-k}Q^{u\varepsilon -(u-k-1)/k}.
\end{align*}
The second conclusion of the lemma now follows.
\end{proof}

\begin{proof}[The proof of Theorem \ref{theorem1.3}]
The first conclusion of Theorem \ref{theorem1.3} follows immediately, on combining those of Corollary \ref{corollary5.2} and Lemma \ref{lemma5.3}.\par

We now turn to the second conclusion of the theorem. According to the hypotheses in play, we have $\Omega=\min\{\omega,\omega'\}$ with $\omega$ and $\omega'$ 
positive numbers as prescribed in the statement of Theorem \ref{theorem1.3}. We assume now that $u>2t+4$ and $1\le U\le P^{\Omega}$. Then, again combining the 
conclusions of  Corollary \ref{corollary5.2} and Lemma \ref{lemma5.3}, we obtain the bound
\[
\int_{\grN(U)}|g(\alpha ;P,R)|^u\d\alpha \ll P^{u-k}U^{\varepsilon -(u-2t-4)/(2k)}.
\]
Thus, by summing over the values $U=2^{-j}P^{\Omega}$ with $j\ge 0$ and $U\ge Q$, we find that
\[
\int_{\grM(P^{\Omega})\setminus \grM(Q)}|g(\alpha ;P,R)|^u\d\alpha \ll P^{u-k}Q^{-\delta},
\]
provided that $\delta <(u-2t-4)/(2k)$. This completes the proof of Theorem~\ref{theorem1.3}.
\end{proof}

\section{Some consequences for Waring's problem}
We now discuss the consequences of Theorem \ref{theorem1.3} for Waring's problem. This requires the introduction of new notation. First we recall the concept of an 
admissible exponent. We refer to a real number $\Delta_s$ as an admissible exponent (for $k$) if it has the property that, whenever $\varepsilon>0$ and $\eta$ is a positive 
number sufficiently small in terms of $\varepsilon$, $k$ and $s$, then whenever $1\le R\le P^\eta$ and $P$ is sufficiently large, one has
\[
\int_0^1|g(\alpha ;P,R)|^s\d\alpha \ll P^{s-k+\Delta_s+\varepsilon}.
\]
Given admissible exponents $\Delta_s$ $(s\ge 0)$ for $k$, we define for each positive number $u$ the admissible exponent for minor arcs
\[
\Del_u^*=\min_{0\le t\le u-2}\bigl( \Del_{u-t}-t\tau(k)\bigr),
\]
in which
\[
\tau(k)=\max_{w\in \mathbb N}\frac{k-2\Del_{2w}}{4w^2}.
\]

\par Equipped with these definitions, it is useful to adopt a convention concerning the appearance of the letters $\varepsilon$ and $R$. If a statement involves the letter $R$, 
either implicitly or explicitly, then it is asserted that for any $\varepsilon>0$ there is a number $\eta>0$ such that the statement holds uniformly for $2\le R\le P^\eta$. Our 
arguments will involve only a finite number of statements, and so there is the option to pass to the smallest of the numbers $\eta$ that arise during the course of the 
argument, and then have all estimates in force with the same positive number $\eta$. Notice that $\eta$ may be assumed to be sufficiently small in terms of $k$, $s$ and 
$\varepsilon$.\par

We recall a key mean value estimate from \cite[Theorem 5.3]{BW2023}. Define the set of minor arcs $\grm(Q)=[0,1)\setminus \grM(Q)$ for $1\le Q\le P^{k/2}$. Suppose that 
$\Delta_s^*$ is an admissible exponent for minor arcs with $\Delta_s^*<0$, and $\theta$ is a positive number with $\theta \le k/2$. Then provided that 
$P^\theta \le Q\le P^{k/2}$, one has
\begin{equation}\label{7.1}
\int_{\grm(Q)}|g(\alpha ;P,R)|^s\d\alpha \ll_\theta P^{s-k}Q^{\varepsilon -2|\Delta_s^*|/k}.
\end{equation}
We now offer a refinement of \cite[Theorem 6.1]{BW2023}.

\begin{theorem}\label{theorem7.1}
Let $k\ge 3$ and $s\ge 2\lfloor k/2\rfloor +5$, and suppose that $\Delta_s^*$ is an admissible exponent for minor arcs with $\Delta_s^*<0$. Let $\nu$ be any positive 
number with
\[
\nu<\min \Bigl\{ \frac{2|\Delta_s^*|}{k},\frac{1}{3k}\Bigr\} .
\]
Then, when $1\le Q\le P^{k/2}$, one has the uniform bound
\[
\int_{\grm(Q)}|g(\alpha ;P,R)|^s\d\alpha \ll P^{s-k}Q^{-\nu}.
\]
\end{theorem}

\begin{proof} In view of the upper bound \eqref{7.1}, it suffices to consider values of $Q$ with $1\le Q\le P^\theta$, where $\theta$ is a fixed positive number sufficiently 
small in terms of $k$ and $s$. Put $t=\lfloor k/2\rfloor$. Then Theorem \ref{theorem1.3} shows that whenever $1\le Q\le P^\theta$ and $s\ge 2t+5$, then
\[
\int_{\grM(P^\theta )\setminus \grM(Q)}|g(\alpha ;P,R)|^s\d\alpha \ll P^{s-k}Q^{-1/(3k)}.
\]
We find from \eqref{7.1} that whenever $s\ge 2t+5$, and $\Del_s^*$ satisfies the hypotheses of the statement of the theorem, then
\[
\int_{\grm(P^\theta)}|g(\alpha ;P,R)|^s\d\alpha \ll P^{s-k}(P^{-\theta})^{2|\Del_s^*|/k-\varepsilon}.
\]
Consequently, on choosing $\theta$ to be small enough and $1\le Q\le P^\theta$, we conclude that
\[
\int_{\grm(Q)}|g(\alpha ;P,R)|^s\d\alpha \ll P^{s-k}Q^{-\nu},
\]
where $\nu$ is any positive number smaller than $\min \{ 2|\Delta_s^*|/k, 1/(3k)\}$. The conclusion of the theorem now follows.
\end{proof}

We remark that essentially the same conclusion is obtained in \cite[Theorem 6.1]{BW2023}, though with the condition $s\ge 2k+3$ in place of the present hypothesis 
$s\ge 2\lfloor k/2\rfloor +5$. Thus our new condition represents an improvement by nearly a factor $2$ over this earlier work.\par

Now define
\[
G_0(k)=\min_{v\ge 2}\Bigl( v+\frac{\Delta_v}{\tau(k)}\Bigr) .
\]
Also, when $s\in \mathbb N$, write $R_{s,k}(n)$ for the number of solutions of the 
equation
\[
x_1^k+\ldots +x_s^k=n,
\]
with $x_i\in \mathbb N$. Finally, define $\Gamma_0(k)$ to be the least integer $s$ such 
that, for all natural numbers $n$ and $q$, the congruence
\[
x_1^k+\ldots +x_s^k\equiv n\mmod{q}
\]
is soluble with $(x_1,q)=1$.\par

\begin{theorem}\label{theorem7.2}
Suppose that $k\ge 3$ and
\[
s\ge \max \left\{ \lfloor G_0(k)\rfloor +1, \Gamma_0(k), 2\lfloor k/2\rfloor +5\right\} .
\]
Then provided that the integer $n$ is sufficiently large in terms of $k$ and $s$, one has $R_{s,k}(n)\gg n^{s/k-1}$.
\end{theorem}

\begin{proof} The desired conclusion follows by the argument of the proof of \cite[Theorem 6.2]{BW2023} by substituting our Theorem \ref{theorem7.1} for 
\cite[Theorem 6.1]{BW2023} throughout.
\end{proof}

Again, this conclusion replaces the constraint $s\ge 2k+3$ in \cite[Theorem 6.2]{BW2023} by the present hypothesis $s\ge 2\lfloor k/2\rfloor +5$.

\bibliographystyle{amsbracket}

\begin{thebibliography}{38}

\bibitem{BW2001}
J. Br\"udern and T. D. Wooley, \emph{On Waring's problem for cubes and smooth Weyl sums}, Proc. London Math. Soc. (3) \textbf{82} (2001), no. 1, 89--109.

\bibitem{BW2023}
J. Br\"udern and T. D. Wooley, \emph{On Waring's problem for larger powers}, J. reine angew. Math. \textbf{805} (2023), 115--142.

\bibitem{BW2024}
J. Br\"udern and T. D. Wooley, \emph{Partitio Numerorum: sums of squares and higher powers}, Funct. Approx. Comment. Math. (in press), doi:10.7169/facm/2165; 
available as arXiv:2402.09537.

\bibitem{Hua1965}
L.-K. Hua, \emph{Additive theory of prime numbers}, American Mathematical Society, Providence, RI, 1965.

\bibitem{MV2007}
H. L. Montgomery and R. C. Vaughan, \emph{Multiplicative number theory I. Classical theory}, Cambridge Stud. Adv. Math. vol. 97, Cambridge University Press, Cambridge, 
2007.

\bibitem{PW2002}
S. T. Parsell and T. D. Wooley, \emph{On pairs of diagonal quintic forms}, Compositio Math. \textbf{131} (2002), no. 1, 61--96.

\bibitem{Shpa}
I. Shparlinski, \emph{Exponential sums over integers without large prime divisors}, 15pp., preprint available as arXiv:2404.10278.

\bibitem{Ten}
G. Tenenbaum,  \emph{Introduction to analytic and probabilistic number theory}, Cambridge University Press, Cambridge, 1995.

\bibitem{Vau1986}
R. C. Vaughan, \emph{On Waring's problem for smaller exponents}, Proc. London Math. Soc. (3) \textbf{52} (1986), no. 3, 445--463.

\bibitem{Vau1989a}
R. C. Vaughan, \emph{A new iterative method in Waring's problem}, Acta Math. \textbf{162} (1989), no. 1-2, 1--71.

\bibitem{Vau1989b}
R. C. Vaughan, \emph{A new iterative method in Waring's problem II}, J. London Math. Soc. (2) \textbf{39} (1989), no. 2, 219--230.

\bibitem{Vau1997}
R. C. Vaughan, \emph{The Hardy-Littlewood method}, 2nd edition, Cambridge University Press, Cambridge, 1997.

\bibitem{VW1991}
R. C. Vaughan and T. D. Wooley, \emph{On Waring's problem: some refinements}, Proc. London Math. Soc. (3) \textbf{63} (1991), no. 1, 35--68.

\bibitem{VW1995}
R. C. Vaughan and T. D.  Wooley, \emph{Further improvements in Waring's problem}, Acta Math. \textbf{174} (1995), no. 2, 147--240.

\bibitem{VW2000}
R. C. Vaughan and T. D. Wooley, \emph{Further improvements in Waring's problem, IV: higher powers}, Acta Arith. \textbf{94} (2000), no. 3, 203--285.

\bibitem{Woo1992}
T. D. Wooley, \emph{Large improvements in Waring's problem}, Ann. of Math. (2) \textbf{135} (1992), no. 1, 131--164.

\bibitem{Woo1995}
T. D. Wooley, \emph{New estimates for smooth Weyl sums}, J. London Math. Soc. (2) \textbf{51} (1995), no. 1, 1--13.

\bibitem{Woo2016}
T. D. Wooley, \emph{On Waring's problem for intermediate powers}, Acta Arith. \textbf{176} (2016), no. 3, 241--247.

\end{thebibliography}
\providecommand{\bysame}{\leavevmode\hbox to3em{\hrulefill}\thinspace}

\end{document}